\documentclass{amsart}

\usepackage{appendix}
\usepackage{amsmath}
\usepackage{amsthm}
\usepackage{amssymb}
\usepackage{amsfonts}
\usepackage{bm}
\usepackage[shortlabels]{enumitem}
\usepackage[bookmarks=true,hyperindex,pdftex,colorlinks,citecolor=red, linkcolor=cyan]{hyperref}
\usepackage{centernot} %for negations
\usepackage{marginnote} % add margin notes
\usepackage{soul} % strike through words
\usepackage[left=2.5cm,right=2.5cm,top=2.5cm,bottom=2.5cm]{geometry}
\usepackage{bbm}
\usepackage[all]{xy}
\usepackage{tikz}
\usepackage{comment}
\usetikzlibrary{arrows}
\usetikzlibrary{shapes,decorations}
\usetikzlibrary{positioning}
\usepackage{xcolor}

\usepackage{mathrsfs}  

\definecolor{darkgreen}{rgb}{.2, .6, .2}

\DeclareMathOperator{\lspan}{span}                          % linear span
\DeclareMathOperator{\supp}{supp}                           % support
\DeclareMathOperator{\diam}{diam}                           % diameter
\DeclareMathOperator{\Lip}{Lip}                             % Lipschitz functions
\newcommand{\N}{\mathbb{N}}             % natural numbers
\newcommand{\M}{\mathbb{M}}             % infinite subsets of natural numbers
\newcommand{\Z}{\mathbb{Z}}             % integer numbers
\newcommand{\R}{\mathbb{R}}             % real numbers
\newcommand{\C}{\mathbb{C}}             % complex numbers
\newcommand{\U}{\mathbb{U}}             % roots of unity
\newcommand{\D}{\mathbb{D}}             % complex unit disc
\newcommand{\B}{\overline{B}}			% closed ball

\newcommand{\F}{\mathcal{F}}                                % Lipschitz-free space
\newcommand{\restricted}{\mathord{\upharpoonright}}

\def\<{\langle}
\def\>{\rangle}
\newcommand{\ep}{\varepsilon}

\theoremstyle{plain}
\newtheorem{theorem}{Theorem}[section]
\newtheorem{lemma}[theorem]{Lemma}
\newtheorem{corollary}[theorem]{Corollary}
\newtheorem{proposition}[theorem]{Proposition}

\newtheorem{maintheorem}{Theorem} % for fancy main theorems with letter numbering

\theoremstyle{definition}
\newtheorem*{definition*}{Definition}
\newtheorem{definition}[theorem]{Definition}
\newtheorem{example}[theorem]{Example}
\newtheorem{question}{Question}

\theoremstyle{remark}
\newtheorem{remark}[theorem]{Remark}

\begin{document}

\title[On the spectrum of Lipschitz operators and composition operators on Lipschitz spaces]{A note on the spectrum of Lipschitz operators and composition operators on Lipschitz spaces}

\author[A. Abbar]{Arafat Abbar}

\author[C. Coine]{Cl\'ement Coine}

\author[C. Petitjean]{Colin Petitjean}

\address[A. Abbar]{LAMA, Univ Gustave Eiffel, Univ Paris Est Creteil, CNRS, F--77447, Marne-la-Vall\'ee, France}
\email{abbar.arafat@gmail.com}

\address[C. Coine]{Normandie Univ, UNICAEN, CNRS, LMNO, 14000 Caen, France}
\email{clement.coine@unicaen.fr}

\address[C. Petitjean]{LAMA, Univ Gustave Eiffel, Univ Paris Est Creteil, CNRS, F--77447, Marne-la-Vall\'ee, France}
\email{colin.petitjean@univ-eiffel.fr}

\date{} % uncomment to remove date from title

% KEYWORDS 
\subjclass[2020]{Primary 47B01, 47A10, 47B07; Secondary 46B20.}

%46B20 Geometry and structure of normed linear spaces
%46B50 Compactness in Banach (or normed) spaces
%47A10 Spectrum, resolvent
%47B01 Operators on Banach spaces
%47B07 Linear operators defined by compactness properties
%54E35: Metric spaces, metrizability
%54E50 Complete metric spaces

\keywords{Spectrum, Point spectrum, Compact operator, Composition operator, Lipschitz-free space}

\maketitle

\begin{abstract}
Fix a metric space $M$ and let $\Lip_0(M)$ be the Banach space of complex-valued Lipschitz functions defined on $M$. A weighted composition operator on $\Lip_0(M)$ is an operator of the kind $wC_f : g \mapsto w \cdot g \circ f$, where $w : M \to \C$ and $f: M \to M$ are any map. When such an operator is bounded, it is actually the adjoint operator of a so-called weighted Lipschitz operator $w\widehat{f}$ acting on the Lipschitz-free space $\F(M)$. 
In this note, we study the spectrum of such operators, with a special emphasize when they are compact. Notably, we obtain a precise description in the non-weighted $w \equiv 1$ case: the spectrum is finite and made of roots of unity. 
\end{abstract}

\section{Introduction}

If $(M,d)$ is a pointed metric space, $0 \in M$ denoting the distinguished point, $\Lip_0(M)$ stands for the vector space of Lipschitz functions $f : M \to \C$ which vanish at the base point. This space naturally becomes a Banach space when equipped with the Lipschitz norm $\Lip(f)$ which is simply the best Lipschitz constant of the map $f$, in other words:
$$
\|f\|_L := \Lip(f) = \sup_{x \neq y \in M} \frac{|f(x)-f(y)|}{d(x,y)}.
$$
It turns out that this space has a (topological) predual which we describe now.  For $x\in M$, we let $\delta(x) \in \Lip_0(M)^*$ be the evaluation functional defined by $\<\delta(x) , f \> = f(x), \ \forall f\in \Lip_0(M).$ 
It is rather easy to see that the map $\delta : M \to \Lip_0(M)^*$ is actually an isometry. Now the \textit{Lipschitz free space over $M$} is defined as follows:
    $$\F(M) := \overline{ \mbox{span}}^{\| \cdot  \|} \delta(M) = \overline{ \mbox{span}}^{\| \cdot  \|}\left \{ \delta(x) \, : \, x \in M  \right \} \subset \Lip_0(M)^*.$$
A fundamental property of $\F(M)$ is the next canonical ``extension'' property: For any Lipschitz map $f$ from $M$ to a complex Banach space $X$ (again vanishing at the base point), there is a unique bounded linear operator $\overline{f}$ from $\F(M)$ to $X$ such that $\|\overline{f}\| = \Lip(f)$ and $f = \overline{f} \circ \delta$. One then obtain as a straightforward consequence that $\F(M)^*$ is linearly isometric to $\Lip_0(M)$. We refer the reader to \cite{GK, Weaver2} for general information about Lipschitz spaces and Lipschitz free spaces. 
\medskip

Thanks to the extension property described above, one also readily deduce that any base point preserving Lipschitz map $f : M \to M$ can be associated with a bounded linear operator $\widehat{f} : \F(M) \to \F(M)$, in such a way  $\|\widehat{f}\| = \Lip(f)$ and the following diagram commutes:
$$\xymatrix{
		M \ar[r]^f \ar[d]_{\delta}  & M \ar[d]^{\delta} \\
		\F(M) \ar[r]_{\widehat{f}} & \F(M).
	}
 $$
Operators of the kind $\widehat{f}$ are often referred to as \textit{Lipschitz operators} (or \textit{linearized Lipschitz operators}). In fact, we will study weighted versions of these operators: For any two maps $f : M \to M$ and $w : M \to \C$, we consider the linear map $w\widehat{f}  : \mbox{span} \left( \delta(M) \right) \to \F(M)$ given by the formula
$$
\ w\widehat{f}\Big(\sum_{i=1}^n a_i\delta(x_i)\Big) = \sum_{i=1}^n a_i w(x_i) \delta(f(x_i)).
$$
When this linear map is continuous, it has a unique extension still denoted by
$
w\widehat{f}  : \F(M) \to \F(M)
$. Such a bounded operator is named \textit{weighted Lipschitz operator} and it turns out that its adjoint corresponds to a weighted composition operator on $\Lip_0(M)$. Indeed, one has for every $g \in \Lip_0(N)$ and $x \in M$:
$$ \<(w\widehat{f})^*(g) , \delta(x) \> = \< g , w\widehat{f}(\delta(x)) \> = \< g , w(x) \delta(f(x))\> = w(x) g(f(x)) = w(x) g \circ f(x).$$
In the sequel, we will write $wC_f : \Lip_0(N) \to \C^M$ for the \textit{weighed composition operator} defined by:
$$ \forall g \in \Lip_0(N), \forall x \in M, \quad  wC_f(g)(x) = w(x) g \circ f(x).$$
A complete description of bounded weighted composition operators on $\Lip_0(M)$ (and as a by-product of bounded weighted Lipschitz operators on $\F(M)$) has recently been given in \cite{GolMa22}; see also \cite{ACP23} and the references therein. For convenience of the reader, let us state some of the content of \cite[Theorem~3.3]{ACP23}:
Let $w : M \to \C$ and $f : M \to M$ be any maps such that $f(0)=0$ or $w(0)=0$. Then the following assertions are equivalent:
\begin{enumerate}[leftmargin=*, itemsep = 4pt]
		\item[$(i)$]  $w \widehat{f}$ extends to a bounded operator from  $\F(M)$ to $\F(M)$;
		\item[$(ii)$] $wC_f$ defines a bounded operator from $\Lip_0(M)$ to $\Lip_0(M)$ (and $wC_f = (w \widehat{f})^*$);
		\item[$(iii)$] $\varphi : x \in M \mapsto w(x)\delta(f(x)) \in \F(M)$ is Lipschitz (and $\overline{\varphi} = w \widehat{f}$);
		\item[$(iv)$] $A:= \displaystyle \sup_{x \neq y} A(x,y) < \infty$ and $B := \displaystyle \sup_{x \neq y} B(x,y) < \infty$, where
\begin{align*}
	A(x,y) &= \frac{1}{d(x,y)}|w(x)d(f(x),0) - w(y)d(f(y),0)|,  \\
	B(x,y) &= \frac{1}{d(x,y)}|w(x)d(f(x),0) - w(y)(d(f(x),0) - d(f(x),f(y))|.
\end{align*}
	\end{enumerate}
\medskip

The goal of this paper is to explore the spectrum  $\sigma(w\widehat{f})$ (and the point spectrum $\sigma_p(w\widehat{f})$) of weighted Lipschitz operators, with clear consequences on the spectrum of weighted composition operators via the usual identification $\sigma(wC_f) = \sigma(w\widehat{f})$. Precisely, a special attention is paid to the case when $w\widehat{f}$ is compact. As a consequence of our investigation, we obtain the next result for the non-weighted case, which comes as a by-product of Theorem~\ref{thm-spectrum} and Theorem~\ref{thmFlatInf}. 

\begin{maintheorem}
	Let $M$ be a pointed metric space and let $\widehat{f}: \F(M) \to \F(M)$ be a compact Lipschitz operator. If $M$ is bounded, or if $f$ is flat at infinity,
	then 
	$$\sigma(\widehat{f})\setminus\{0\} = \bigcup_{n \in A} \U_n,$$ 
	where $\U_n = \{ z \in \C \; : \; z^n = 1 \}$ and $A$ is the set of all natural numbers $n$ such that $f$ has a periodic point of order $n$. 
\end{maintheorem}

A full characterization of compact Lipschitz operators ($w \equiv 1$) is provided in \cite{ACP21}; it is the natural continuation of the works in \cite{Vargas2, Vargas1}. For compactness results in the weighted case, we refer the reader to \cite{ACP23, GolMa22} (and references therein). Notably, \cite[Corollary 4.2]{ACP23} states that a weighted Lipschitz operator is compact if and only if it is weakly compact. For the sake of readability, we postpone the presentation of the precise characterizations of compactness to later in the sequel. Let us point out that there is not much work on the spectrum of these operators in the literature. With respect to the spectrum of composition operators on Lipschitz spaces, we wish to mention the early work of H. Kamowitz \cite{Kamo2}, H. Kamowitz and S. Scheinberg \cite{Kamo} and also the more recent work of A. Jim\'{e}nez-Vargas and M. Villegas-Vallecillos in \cite{Vargas1}. In the aforementioned papers, the authors work with $\Lip(M)$, the Banach space of bounded Lipschitz functions $g : M \to \C$, equipped with the norm $\|g\|_{\Lip}:=\max\{\|g\|_{\infty} , \Lip(g)\}$.
Recall that every $\Lip$ space is, in every meaningful sense, a $\Lip_0$ space (see \cite[Section~2.2]{Weaver2}). So it is in fact more general to work with $\Lip_0$ spaces as one can derive the corresponding
results for Lip spaces; see \cite[Section~1.2]{ACP23} for more details. 
On the other hand, if $M$ is bounded then $\Lip_0(M)$ is isomorphic to a codimension one ideal of $\Lip(M)$. 
\bigskip

\noindent\textbf{Notations and background.}
\medskip

- For roots of unity, we will use the notation $\U_n = \{ z \in \C \; : \; z^n = 1 \}$. 
\smallskip

- If $w : M \to \C$ is a weight map, we let $$\text{coz}(w) := w^{-1}(\C^*) = \{ x\in M \mid w(x) \neq 0 \}.$$

- If 
$g : M \to \C$ is a bounded map, we let $\|g\|_{\infty}:= \sup_{x \in M} |g(x)|$.
\smallskip

- If $X$ is a Banach space over $\C$, we let $X^*$ be its topological dual and $B_X$ (resp. $S_X$) be its unit ball (resp. its unit sphere). If $Y$ is another Banach space, we let $\mathcal{L}(X,Y)$ be the space of bounded linear operators $T : X\to Y$. If $X=Y$, we simply write $\mathcal{L}(X)$.
\smallskip

- If $T\in \mathcal{L}(X)$ then the spectrum of $T$, denoted $\sigma(T)$, is the set of all $\lambda \in \C$ for which the operator $T - \lambda Id$ is not bijective. Next, the point spectrum of $T$, $\sigma_p(T)$, is the set of all eigenvalues of $T$. That is, the set of all $\lambda \in \C$ such that $\ker (T - \lambda Id) \neq \{0\}$. It is clear that $\sigma_p(T) \subset \sigma(T)$, and it is well-known that $\sigma(T) \setminus \{0\} \subset \sigma_p(T)$  when $T$ is compact. Moreover, if $T$ is compact, its spectrum is countable and zero is the only possible accumulation point. Moreover, if $T^* : X^* \to X^*$ is the adjoint of $T$, then $\sigma (T^{*})=\sigma (T)$.
\smallskip

- We will only consider \textbf{complete} pointed metric spaces $(M,d)$. There is not much loss of generality since the Lipschitz free space over $M$ is always isometrically isomorphic to the free space over its completion. As usual, $B(x,r)$ (respectively $\overline{B}(x,r)$) denotes the open ball (respectively closed ball) centered at $x \in M$ and of radius $r \geq 0$. 
\smallskip

- A map $f : M \to M$ is said to be uniformly locally flat if for every $\ep >0$, there exists $\delta >0$ such that for all $x,y \in M$:
$$ d(x,y) \leq  \delta \implies d(f(x),f(y)) \leq \ep d(x,y).$$
\smallskip

- If $N$ is a subset of $M$, then $\F(N\cup\{0\})$ is naturally (isometrically) identified with the subspace 
$$\F_M(N):=\overline{\lspan}\{\delta(x) : x \in N\} \subset \F(M).$$
\smallskip

- Finally, if $\gamma \in \F(M)$ then $\supp(\gamma)$ stands for the support of $\gamma$, that is the smallest closed subset $N\subset M$ such that $\gamma \in \F_M(N)$. In particular, $0$ is never an isolated point in $\supp(\gamma)$. More information about supports in Lipschitz-free spaces can be found in  \cite{support1, APPP2019} (see also \cite[Section~2.3]{ACP23} for a discussion around the complex scalar case).

\section{Periodic points and point spectrum of weighted operators} \label{sectionspectrum}

Throughout this section, we assume that $w \widehat{f} : \F(M) \to \F(M)$ and $wC_f : \Lip_0(M) \to \Lip_0(M)$ are bounded operators. In fact, we will mainly deal with  $w \widehat{f}$, and state some consequences later on the spectrum of composition operators.

\begin{definition}
A point $x\in M$ is a periodic point of order $n \in \N$ (or equivalently with period $n$) of a map $f : M \to M $ if $f^n(x) = x$ and $f^k(x) \neq x$ for $1\leq k \leq n-1$. We denote by $\mathrm{Per}_{n}(f)$ the set of periodic points of order $n$ of $f$.
\end{definition}

\begin{lemma}
	\label{lemmaPeriodic1}
	Assume that $f$ has a periodic point $x$ of order $n$ and that $w(f^k(x)) \neq 0$ for every $0 \leq k \leq n-1$. If $\lambda \in \mathbb{C}$ is such that $\lambda^n = \prod_{k=0}^{n-1} w(f^k(x))$, then $\lambda \in \sigma_p(w\widehat{f})$. In particular, if $w \equiv 1$ then $\U_n \subset \sigma_p(\widehat{f})$. 
\end{lemma}

\begin{proof}
	For simplicity, we will write $x_i = f^i(x)$ for every $i \in \N$. Since $x$ is periodic point of order $n$, $x_n = x$ and the set $S = \{ x_i \in M \; : \; 0 \leq i \leq n-1  \}$
	contains exactly $n$ distinct points of $M$. Let us fix $\lambda \in \mathbb{C}$ such that $\lambda^n = \prod_{k=0}^{n-1} w(x_k)$. Define $(a_k)_{k=0}^n$ recursively by $a_0=1$ and $a_{i+1} =(\lambda^{-1})w(x_i) a_i$. Note that 
$$
    a_n = (\lambda^{-1})a_{n-1}w(x_{n-1}) = \cdots = (\lambda^{-n})\prod_{k=0}^{n-1} w(x_k) = 1.
$$
 Next, we consider $\displaystyle \gamma = \sum_{i=0}^{n-1} a_i \delta(x_i) \in \F(M) .$ Some direct computations show that 
	\begin{align*}
	w\widehat{f}(\gamma) & = \sum_{i=0}^{n-1} a_i w(x_i)\delta(f(x_i))
	  = \sum_{i=0}^{n-1} a_i w(x_i) \delta(x_{i+1}) = \sum_{i=1}^n a_{i-1}w(x_{i-1}) \delta(x_i) = \sum_{i=1}^n \lambda a_{i} \delta(x_i)  = \lambda \gamma.
	\end{align*}
\end{proof}

The next result can be seen as a partial converse of the previous lemma.

\begin{lemma}
	\label{lemmaPeriodic2}
	Assume that there exist $\lambda \in \C$, $n \in \N$ and $\gamma \in \F(M)$ such that $|\supp(\gamma)| = n$ and $w\widehat{f}(\gamma) = \lambda \gamma$. We write $\supp \gamma = \{x_1 , \ldots , x_n\}$. 
    \begin{itemize}
        \item 
        If $\lambda \neq 0$, then each $x_i$ is a periodic point of $f$ whose order $m_i$ is smaller than $n$. Moreover, $\lambda^{m_i} = \prod_{j=0}^{m_i-1} w(f^j(x_i))$.
        \item If $f$ is injective, $n \geq 2$ or $f(x_1) \neq 0$, and $w(x_i) \neq 0$ for every $i \in \{1 , \ldots , n\}$, then $\lambda \neq 0$. 
    \end{itemize}
In particular if $w \equiv 1$ then $\lambda \in \U_n \cup \{0\}$ and $\lambda \neq 0$ whenever $f$ is injective.
\end{lemma}

\begin{proof}
	There exist $a_{x_1},\ldots,a_{x_n} \in\C^{\ast}$ such that $\gamma=\sum_{i=1}^{n}a_{x_i}\delta(x_i).$
	By assumption, $w\widehat{f}(\gamma) = \lambda \gamma$, which reads
	\begin{equation} 
	\label{eq:1}
\sum_{i=1}^{n}a_{x_i}w(x_i)\delta(f(x_i))=\lambda\sum_{i=1}^{n} a_{x_i}\delta(x_i).  
	\end{equation}
This yields that $\supp(w\widehat{f}(\gamma)) = \supp ( \gamma)$, unless $\lambda = 0$. Therefore, if $\lambda \neq 0$ then  $w(x_i) \neq 0$ for every $i \in \{1 , \ldots , n\}$ and $\{ f(x_i):\, 1\leq i\leq n\}= \{ x_i:\, 1\leq i\leq n\}$. 
Thus $f$ is a permutation of the set $\{ x_i:\, 1\leq i\leq n\}$, and so we can decompose $f$ as a product of cycles with disjoint supports. Hence there exist $y_1 , \ldots , y_k \in M$ and positive integers $n_1 , \ldots , n_k $ such that $n = n_1 + \ldots + n_k$, each $y_i$ is a periodic point of order $n_i$, and 
$$\{ x_i:\, 1\leq i\leq n\} = \bigsqcup_{i=1}^k \{f^j(y_i) : \, 0 \leq j \leq n_i-1 \}.$$
Now if $x_i$ is any point of $\text{supp}(\gamma)$, its orbit by $f$ is included in the orbit of some $y_j$. In particular, $x_i$ is a periodic point of $f$ and its order, denoted by $m_i$, divides $n_j$. It follows from \eqref{eq:1} that
\begin{align*}
\lambda \sum_{j=0}^{m_i-1} a_{f^j(x_i)} \delta(f^j(x_i))
& = \sum_{j=0}^{m_i-1} a_{f^j(x_i)} w(f^j(x_i)) \delta(f^{j+1}(x_i)) =\sum_{j=1}^{m_i} a_{f^{j-1}(x_i)} w(f^{j-1}(x_i)) \delta(f^j(x_i)),   
\end{align*}
which in turn yields
$$
\lambda a_{f^j(x_i)} = a_{f^{j-1}(x_i)} w(f^{j-1}(x_i)).
$$
Finally, denoting $p_i = \prod_{j=0}^{m_i-1} a_{f^j(x_i)}$, by taking the product over $j$ in the last equality, we obtain:
$$
\lambda^{m_i} p_i = p_i \prod_{j=0}^{m_i-1} w(f^j(x_i)).
$$
We deduce that $\lambda^{m_i} = \prod_{j=0}^{m_i-1} w(f^j(x_i))$.
\medskip
		
Finally, if $f$ is injective and $w(x_i) \neq 0$ for every $i \in \{1 , \ldots , n\}$, we prove that $\lambda\neq0$. Arguing by contradiction, assume that $\lambda=0$. By equation~\eqref{eq:1}, we obtain
$$\sum_{i=1}^{n}a_{x_i}w(x_i)\delta(f(x_i))=0.$$
Since $f$ is injective, the family $\{ \delta(f(x_i)): 1\leq i\leq n\}$ is linearly independent.  We deduce that $a_{x_i} =0$ for every $i \in \{1 , \ldots , n\}$, that is $\gamma = 0$, a contradiction.
\end{proof}

The following lemma concerns preservation of supports under weighted Lipschitz operators. It is the straightforward extension of Proposition 2.1 in \cite{GPP22}.

\begin{lemma}\label{preservsupplemma}
For any $\gamma \in \F(M)$,
$$ \supp\left( w\widehat{f}(\gamma) \right) \subset \overline{f\left(\supp(\gamma) \cap \text{coz} ~ w \right)}.$$
\end{lemma}

\begin{proof}
Let $K \subset M$. It is clear that $w\widehat{f}(\lspan \delta(K)) \subset  \lspan \delta(f(K \cap \text{coz} ~ w))$. On the one hand, $\lspan \delta(f(K \cap \text{coz} ~ w ))$ is dense in the closed space $\F\big(\overline{f(K \cap \text{coz} ~ w )}\big)$. On the other hand, $w\widehat{f}(\lspan \delta(K))$ is dense into the closure of $w\widehat{f}(\F(K))$, which implies that 
\begin{equation}\label{eq:2}
    \overline{w\widehat{f}(\F(K))} \subset \F\left(\overline{f(K \cap \text{coz} ~ w )}\right).
\end{equation}
Let $\gamma \in \F(M)$. By the very definition of the support, $\gamma \in \F(\supp(\gamma))$ and so $w\widehat{f}(\gamma) \in w\widehat{f}\big(  \F(\supp(\gamma)) \big)$. Now the inclusion \eqref{eq:2} implies that $w\widehat{f}(\gamma) \in \F\left( \overline{f\left(\supp(\gamma) \cap \text{coz} ~ w \right)} \right)$. We deduce that 
$$ \supp\left( w\widehat{f}(\gamma) \right) \subset \overline{f\left(\supp(\gamma) \cap \text{coz} ~ w \right)} . $$
\end{proof}

We are now ready to prove the main result of this section, which we shall apply later to some compact Lipschitz operators. Recall that a map $f : M \to M$ is $(r,\ep)$-flat, where $\ep, r >0$, if 
$$\forall x \neq y \in M, \quad d(x,y)\leq r \implies \dfrac{d(f(x),f(y))}{d(x,y)} \leq \ep.$$
In the next theroem, $A$ denotes the set of all natural numbers $n$ such that $f$ has a nonzero periodic point of order $n$. Recall that $\mathrm{Per}_n(f) \subset M$ stands for the set of nonzero periodic points of order $n$.

\begin{theorem}
	\label{thm-spectrum}
    Assume that $f(M)$ is totally bounded and that $f$ is $(r,\ep)$-flat for some $\ep \in (0,1)$ and $r >0$. Then any eigenvector of $w\widehat{f}$  associated with a nonzero eigenvalue has finite support, and
    $$\sigma_p(w\widehat{f})\setminus\{0\} = \bigcup_{n \in A} \Big\{\lambda \in \C : \exists x \in \mathrm{Per}_n(f), \;  \lambda ^n = \prod_{i=1}^n w\big(f^i(x)\big) \Big\}.$$
    In particular, if $w \equiv 1$ then $\sigma_p(\widehat{f})\setminus\{0\} = \bigcup_{n \in A}\U_n$.
\end{theorem}

\begin{proof}
	The inclusion ``$\supset$'' follows from Lemma~\ref{lemmaPeriodic1}. Conversely, let $\lambda \in \sigma_p(w\widehat{f})\setminus\{0\}$ and fix $\gamma \in \F(M)\setminus\{0\}$ such that $w\widehat{f}(\gamma) = \lambda \gamma$. If $\gamma$ is finitely supported, then we simply apply Lemma~\ref{lemmaPeriodic2} to obtain the desired conclusion. If the support of $\gamma$ is infinite, we will prove that $w\widehat{f}(\gamma) = \lambda \gamma$ cannot hold with $\lambda \neq 0$. By assumption, $f(M)$ is totally bounded so there exists $x_1, \ldots , x_N \in M$ such that $f(M) \subset \bigcup_{i=1}^N \overline{B}(x_i,\frac{r}{2})$. 
	As $f$ is $(r,\ep)$-flat, it is readily seen that $f^n(\overline{B}(x_i,\frac{r}{2})) \subset \overline{B}(f^n(x_i),\ep^n \frac{r}{2})$ for $n \in \N$  and $i \in \{1, \ldots , N\}$. 
	Since $\overline{f(M)}$ is compact, there exists an infinite subset $\M$ of $\N$ such that the sequences $(f^{n}(x_i))_{n \in \M}$, $1 \leq i \leq N$, converge to some $y_i \in M$.
	Now the equality $(w\widehat{f})^n(\gamma) = \lambda^n \gamma$, with $\lambda \neq 0$, yields that $\supp((w\widehat{f})^n(\gamma))  = \supp(\gamma)$ for every $n$. Consequently, thanks to Lemma~\ref{preservsupplemma}, we have the following inclusions for every $n \in \M$:
	\begin{align*}
	\supp(\gamma) &= \supp((w\widehat{f})^n(\gamma))  \subset \overline{f^n(\supp(\gamma))} 
	\subset \overline{f^n\Big(\bigcup_{i=1}^N \overline{B}(x_i,\frac{r}{2})\Big)} \subset \bigcup_{i=1}^N \overline{B}(f^n(x_i),\ep^n \frac{r}{2}) \\
	&\subset \bigcup_{i=1}^N \overline{B}\Big(y_i , \ep^n \frac{r}{2} + d(y_i , f^n(x_i)) \Big).
	\end{align*}
	Since $\ep \in (0,1)$ we have that:
	$$\underset{n \in \M}{\lim\limits_{n \to \infty}}  \ep^n \frac{r}{2} + d(y_i , f^n(x_i))  = 0.$$ Therefore, we deduce that $\supp(\gamma) \subset \{y_1 , \ldots , y_{N}\}$, which is a contradiction. So if $\supp(\gamma)$ is infinite and $\widehat{f}(\gamma) = \lambda \gamma$, then $\lambda=0$.
\end{proof}

\begin{remark}
    Under the assumptions of the previous theorem, one can describe the eigenspace associated to an eigenvalue $\lambda \neq 0$ as follows. For every $n \in \N$, let $S_n = \{x \in  \mathrm{Per}_n(f) : \lambda^n = \prod_{i=0}^{n-1} w\big(f^i(x)\big) \}$. Note that $S_n$ might be empty for some but not all values of $n \in \N$. Then the eigenspace associated to $\lambda$ is the sum (for $n \in \N$) of the vector spaces generated by the elements of the form:
    $$ \sum_{i=0}^{n-1}  a_i \delta(f^i(x))  \text{ where } x \in S_n, \; a_0=1 \text{ and } a_{i+1} = (\lambda^{-1})w(x_i)a_i.$$
\end{remark}

\begin{remark} \label{remark:injectivity}
	On the one hand, it is straightforward to see that if $w(x) = 0$ or $f(x)=0$ for some $x\neq 0 \in M$, then $0 \in \sigma_p(w \widehat{f})$. On the other hand, to our knowledge, there is no characterization of injective (nor surjective) weighted Lipschitz operators in terms of properties of $w$ and $f$.
	In \cite[Corollary~2.7]{GPP22}, it is proved that, for a bounded $M$, a Lipschitz operator $\widehat{f} : \F(M) \to \F(N)$ is injective if and only if $f$ preserves supports in the following sense: \[\forall \mu \in \F(M), \quad \supp(\widehat{f}(\mu))=\overline{f(\supp(\mu))}.\]
	This is satisfied by some Lipschitz maps which are locally bi-Lipschitz \cite[Proposition~3.5]{GPP22}. In particular, this provides metric conditions which ensures that 0 is not in the point spectrum of $\widehat{f}$. 
\end{remark}

\section{Consequences when \texorpdfstring{$w \equiv 1$}{w = 1}}

\subsection{For bounded $M$}

We recall that, if $M$ is bounded, then $\widehat{f} : \F(M) \to \F(M)$ is compact if and only $f$ is uniformly locally flat and $f(M)$ is totally bounded (see \cite[Theorem~A]{ACP21}). In particular, the next corollary can be applied to these compact Lipschitz operators. 

\begin{corollary} 
	\label{cor-spectrum-compact}
	Let $M$ be a complete metric space and $f : M \to M$ be a base point preserving Lipschitz map. Let $A$ be the set of all natural numbers $n$ such that $f$ has a periodic point of order $n$. If 
 $f(M)$ is totally bounded and $f$ is uniformly locally flat, then $A$ is finite and moreover
	$$\sigma(C_f)\setminus\{0\} = \sigma(\widehat{f})\setminus\{0\} = \bigcup_{n \in A} \U_n .$$ 
\end{corollary}

\begin{proof}
	Since $f$ is uniformly locally flat, for every $\ep \in (0,1)$ there exists $r>0$ such that $f$ is $(r, \ep)$-flat. Next, for every compact operator $T$, we have
	\begin{itemize}
		\item $\sigma_p(T) \setminus \{0\} = \sigma(T) \setminus \{0\}$, 
		\item if $\sigma(T)\setminus\{0\}$ is infinite then it must be a sequence converging to 0.
	\end{itemize}
By Theorem~\ref{thm-spectrum}, this implies that $A$ is finite and $\sigma(\widehat{f})\setminus\{0\} = \bigcup_{n \in A} \U_n $.
\end{proof}

Not surprisingly, if $\widehat{f}$ is not compact, then its spectrum can contain any point in $\C$.
\begin{example}
	Let $\lambda \in \C \setminus \{0\}$ and $M=\N \cup \{0\}$ equipped with the metric $d_\lambda$ defined by 
	$$d_\lambda(k,0)=\dfrac{1}{2^k|\lambda|^k}, \quad \text{and }\quad d_\lambda(k,k')=d_\lambda(k,0)+d_\lambda(k',0).$$ Let $f : M \to M$ be the map given by $f(0)=0$ and $f(n)=n-1$ if $n \in \N$.  It is clear that $f$ is Lipschitz. Let $\gamma=\displaystyle \sum_{k=1}^{+\infty} \lambda^k \delta(k).$ One has 
	$\|\gamma\|\leq \displaystyle \sum_{k=1}^{+\infty}\dfrac{1}{2^k},$
	and 
	$\widehat{f}(\gamma)=\displaystyle \sum_{k=1}^{+\infty}\lambda^k\delta(k-1)=\lambda \gamma$.
	Hence $\lambda\in \sigma_p(\widehat{f})$.
\end{example}

Let us point out that if $M$ is infinite and connected, then the spectrum of a compact Lipschitz operators $\widehat{f}$ is actually reduced to $\{0\}$. This should be compared to 
\cite[Theorem~2]{Kamo} (when $M$ is compact) and to
\cite[Theorem~1.4]{Vargas1} where it is proved that $\sigma(C_f) = \{0,1\}$ whenever $M$ is infinite and connected and $C_f$ is a compact composition operator acting on $\Lip(M)$.

\begin{corollary}\label{spec_zero}
Let $M$ be a connected bounded metric space and let $f : M \to M$ be a base point preserving Lipschitz map such that $\widehat{f}$ is compact. Then $\sigma(C_f) = \sigma(\widehat{f})=\{0\}$. 
\end{corollary}

\begin{proof}
Following \cite{Kamo2}, we consider the set $K := \cap_{n\geq 1} \overline{f^n(M)}$. Note that since $f(0)=0$, we have $0\in K$. Since $f(M)$ is totally bounded, the subsets $\overline{f^n(M)} \subset M, n\geq 1$, are compact. Then, by a similar reasoning than the one in the proof of Theorem~\ref{thm-spectrum}, we get that $K$ is finite. But notice that $(\overline{f^n(M)})_{n\geq 1}$ is a decreasing sequence of connected compact sets, so their intersection is connected and it follows that $K = \{0\}$. Finally, assume that there exists $\lambda \in \sigma_p(\widehat{f}) \setminus \{0\}$. Then, according to the previous theorem, $\lambda \in \U_n$ for some $n \in A$. Next, if $\gamma$ is an eigenvector associated to $\lambda$, then $\gamma$ is finitely supported by  Theorem~\ref{thm-spectrum}. Moreover, the proof of Lemma~\ref{lemmaPeriodic2} shows that $\supp(\gamma)$ can be written as $\bigsqcup_{i=1}^k \{f^j(y_i) : \, 0 \leq j \leq n_i-1 \}$ where each $y_i$ is a periodic point with period $n_i$.
Consequently,  
$\supp\left(\gamma \right) \subset K$ and so $\gamma = 0$, a contradiction. This shows that $\sigma(\widehat{f}) = \{0\}$.
\end{proof}

Notice that the discrepancy between $\{0,1\}$ and $\{0\}$ is explained by the fact that $\Lip(M)$ contains constant functions whereas the only constant function in $\Lip_0(M)$ is $0$. For instance, in \cite[Theorem 2]{Kamo}, the authors show that for the composition operator $C_f : g \in \Lip(M) \mapsto g \circ f \in \Lip(M)$ on a compact connected metric space $M$, the eigenspace associated to the eigenvalue $1$ consists only in the constant Lipschitz functions. Therefore $1$ is not an eigenvalue of $C_f$ when restricted to $\Lip_0(M)$, and hence is not an eigenvalue of $\widehat{f}$.\\

Finally, we identify different conditions under which the point spectrum is reduced to $\{0\}$ (but without assuming that $\widehat{f}$ is compact). Recall that the spectral radius of a bounded operator $T: X \to X$ is defined by $r(T) : = \sup_{\lambda \in \sigma(T)} |\lambda| $. Now Gelfand's formula, also known as the spectral radius formula, gives that $r(T) = \lim\limits_{n \to +\infty} \|T^n\|^{\frac 1 n} = \inf_{n \in \N} \|T^n\|^{\frac 1 n}$.

\begin{proposition}
    Let $M$ be a bounded complete metric space and let $f : M \to M$ be a base point preserving Lipschitz map such that $r(\widehat{f})<1$ (e.g. $\Lip(f) < 1$). Then $\sigma_p(\widehat{f}) \subset \{0\}$.
\end{proposition}

\begin{proof}
    Suppose that $\widehat{f}(\gamma) = \lambda \gamma$ with $\lambda \neq 0$ and $\gamma \neq 0$. Since $r(\widehat{f})<1$, there exists $N \in \N$ such that $\Lip(f^N)<1$. Moreover $\widehat{f^k}(\gamma) = \lambda^k \gamma$ for every integer $k$, therefore: 
    $$\forall k \in \N, \quad \supp(\gamma) = \supp(\widehat{f^{Nk}})(\gamma)) \subset \overline{f^{Nk}(M)} \subset \B(0,\Lip(f^N)^k \diam(M)). $$ 
    We deduce that $\supp(\gamma) \subset \bigcap_{k \in \N} \B\big(0,\Lip(f^N)^{k}\diam(M)\big) = \{0\}$. This implies that $\gamma = 0$, a contradiction.
\end{proof}

\begin{example}\label{Exshift} The previous result does not hold true for unbounded metric spaces. Indeed, let $M = \mathbb{N} \cup \{ 0 \}$ equipped with the following metric: if $n \neq m\in \mathbb{N}$, $d(0,n) = 2^n$ and $d(n,m) = 2^n+2^m$. It follows from \cite[Proposition 1.6]{ACP20} 
that $\F(M)$ is linearly isometric to $\ell_1(\N)$ via $e_n \mapsto 2^{-n}\delta(n)$,  where $(e_n)_n$ is the usual Schauder basis of $\ell_1(\N)$. Let $f : M \to M$ be given by $f(0)=0$ and $f(n) = n-1$ for $n\in \mathbb{N}$. Then, $\widehat{f}$ is conjugate to the operator $T : \ell_1 \to \ell_1$ given by
    $$
    \forall n \in \mathbb{N}, \ Te_n = \dfrac{1}{2} e_{n-1},
    $$
    that is, $T = \frac{1}{2}S$ where $S$ is the backward shift operator. It is well known that $\sigma_p(S) = D(0,1)$ and $\sigma(S) = \overline{D}(0,1)$. Thus $\sigma_p(\widehat{f}) = \sigma_p(T) = D(0,1/2)$ and $\sigma(\widehat{f}) = \sigma(T) = \overline{D}(0,1/2)$. In particular, if the metric space is unbounded and $\Lip(f) < 1$, the point spectrum can contain nonzero elements.
\end{example}

\subsection{For unbounded $M$.}
In the case when the metric space $M$ is unbounded, we do not know whether Corollary~\ref{cor-spectrum-compact} holds true. However, partial answers are given below.

\begin{proposition}\label{spectrumunboundedcase}
Let $M$ be a complete unbounded metric space and $f : M \to M$ be a base point preserving Lipschitz map.
Assume that $f$ satisfies the following properties:
\begin{itemize}
\item[$(i)$] $f$ is uniformly locally flat;
\item[$(ii)$] $\underset{d(x,y) \to +\infty}{\lim} \ \dfrac{d(f(x),f(y))}{d(x,y)}=0$;
\item[$(iii)$] The spectral radius of $\widehat{f}$ satisfies $r(\widehat{f})<1$.
\end{itemize}
Then $\sigma(\widehat{f}) = \{0\}$.
\end{proposition}

\begin{proof}
Assume first that $\alpha := \Lip(f) < 1$. Proving that $\sigma(\widehat{f}) = \{0\}$ is equivalent to proving that $r(\widehat{f})=0$. By the spectral radius formula, we have
$$
r(\widehat{f}) = \lim_n \| \widehat{f}^n \|^{\frac 1 n} = \lim_n \Lip(f^n)^{\frac 1 n}.
$$
Fix $\varepsilon \in (0,1)$ and let us prove that $r(\widehat{f}) \leq \varepsilon$. We can assume that $f^n \neq 0$ for every $n$, otherwise the result is trivial. For every $n\in \mathbb{N}$, there exist $x_n \neq y_n \in M$ such that
\begin{equation*}
\Lip(f^n) \leq \dfrac{d(f^n(x_n), f^n(y_n))}{d(x_n, y_n)} + \varepsilon^n.
\end{equation*}
By the property $(ii)$, there exists $R>0$ such that
$$
\forall (u,v)\in M^2, \ d(u,v)>R \implies \dfrac{d(f(u), f(v))}{d(u,v)}<\varepsilon.
$$
Next, since $f$ is uniformly locally flat, there exists $r \in (0 , R)$ such that 
$$
\forall (u,v)\in M^2, \ d(u,v)<r \implies \dfrac{d(f(u), f(v))}{d(u,v)}<\varepsilon.
$$
We now need the following three observations. Fix $n\in \N$ and let, for $k \in \N$, $$F_k = \dfrac{d(f^{k+1}(x_n), f^{k+1}(y_n))}{d(f^k(x_n), f^k(y_n))}.$$ 
\begin{itemize}[leftmargin=*]
        \item If, for some $k\in \mathbb{N}$,
        $
        d(f^k(x_n),f^k(y_n)) > R,
        $
        then for every $k' \in \{ 0, \ldots, k \}$:
        $$
        d(f^{k'}(x_n),f^{k'}(y_n)) \geq  \dfrac{d(f^k(x_n),f^k(y_n))}{\alpha^{k-k'}} > R \quad  \text{ and } \quad F_{k'} < \ep.
        $$
        \item If for some $k\in \mathbb{N}$,
        $
        d(f^k(x_n),f^k(y_n)) < r,
        $
        then for every $k' \geq k$,
        $$
        d(f^{k'}(x_n),f^{k'}(y_n)) \leq \alpha^{k'-k} d(f^k(x_n),f^k(y_n)) < r \;  \text{ and } \; F_{k'}< \ep.
        $$
        \item Let us fix $N \in \N$ such that $\alpha^N < \frac{r}{R}$. Notice that $N$ does not depend on $n$. Then, there are at most $N$ integer $k$'s such that $d(f^k(x_n), f^k(y_n)) \in [r,R]$. Indeed, if $r\leq d(f^k(x_n), f^k(y_n)) \leq R$, then for every $k'\geq N$:
        $$
        d(f^{k+k'}(x_n), f^{k+k'}(y_n)) \leq \alpha^{k'} d(f^k(x_n), f^k(y_n)) < r.
        $$
        \end{itemize}
\medskip

We now consider $n>N$ and  $I = \{k \in \{0,\ldots ,n-1\} \,: \, d(f^k(x_n), f^k(y_n)) \in [r,R] \}$.
According to the latter observations, $|I|\leq N$ and we have
    \begin{align*}
      \Lip(f^n)
         \leq \prod_{k=0}^{n-1} F_k + \varepsilon^n
         \leq \Big(\underset{k \in I}{\prod_{k=0}^{n-1}} F_k \Big) \Big(\underset{k \not\in I}{\prod_{k=0}^{n-1}} F_k \Big) + \varepsilon^n 
         \leq \alpha^{N} \times \varepsilon^{n-N}+ \varepsilon^n
         \leq 2\varepsilon^{n-N}.
    \end{align*}
    Hence
    $$
    \Lip(f^n)^{1/n} \leq 2^{1/n} \varepsilon^{1-N/n}.
    $$
This yields
    $
    \limsup_n \Lip(f^n)^{1/n} \leq \varepsilon,
    $
and gives the result when $\Lip(f) < 1$.
\medskip

In the general case, the spectral radius formula $r(\widehat{f}) = \lim_n \Lip(f^n)^{1/n}<1$ gives the existence of $n_0\in \N$ such that $\Lip(f^{n_0})^{1/n_0} < 1$ and hence $\Lip(f^{n_0}) < 1$. Next, the inequality
$$
\forall (u,v)\in M^2, \ \dfrac{d(f^{n_0}(u), f^{n_0}(v))}{d(u,v)} \leq \Lip(f)^{n_0-1} \dfrac{d(f(u), f(v))}{d(u,v)}
$$
shows that $f^{n_0}$ satisfies the properties $(i)$ and $(ii)$. Hence, by the first part of the proof, $\sigma(\widehat{f^{n_0}}) = \{ 0 \}$. But this readily implies that $\sigma(\widehat{f}) = \{ 0 \}$.
\end{proof}

\begin{remark}
In particular, if $\widehat{f}$ is compact with $r(\widehat{f})<1$, then $\sigma(\widehat{f}) = \{ 0 \}$. Indeed, in that case, by \cite[Theorem~A]{ACP21} and \cite[Remark~2.7]{ACP21}, $f$ satisfies properties $(i)$ and $(ii)$ in Proposition \ref{spectrumunboundedcase}.
\end{remark}

Our next aim is to obtain a structural result identical to Corollary~\ref{cor-spectrum-compact}, for a general $M$, but under a stronger assumption on $f$. Recall that $f : M \to M$ is said to be \textit{radially flat} whenever $$\ \  \lim\limits_{d(x,0) \to \infty} \dfrac{d(f(x),0)}{d(x,0)} =0. $$

\begin{lemma} \label{lemma:Rad}
If $f : M \to M$ is radially flat then there exists $R_0>0$ such that for every $R \geq R_0$, $f(\B(0,R)) \subset \B(0,\frac{R}{2})$.  
\end{lemma}

\begin{proof}
By the very definition of radially flat, there exists $r_0 >0$ such that one has for every $x \in M$:
$$d(x,0) \geq r_0 \implies d(f(x) , 0) \leq \frac{1}{2} d(x,0).$$
Now let $r_1 = \max(r_0 \Lip(f) , r_0)$. Notice that $f(\B(0,r_1)) \subset \B(0,r_1)$. Indeed,  if $d(x,0) \leq r_0$ then $d(f(x),0) \leq \Lip(f)d(x,0) \leq r_1$. At the same time, if $r_0 < d(x,0) \leq r_1$ then $d(f(x),0) \leq \frac{1}{2} d(x,0) < r_1$. Now let $R_0 = 2r_1$ and $R \geq R_0$. In order to deduce that $f(\B(0,R)) \subset \B(0,\frac{R}{2})$, we distinguish two cases again:
\begin{itemize}
	\item If $r_1 < d(x,0) \leq R$, then $d(x,0) >  r_0$ so that $d(f(x),0) \leq \frac{1}{2} d(x,0) \leq \frac{R}{2}$;
	\item If  $d(x,0) \leq r_1$, then $d(f(x),0) \leq  r_1 =\frac{R_0}{2} \leq \frac{R}{2}$. 
\end{itemize}
\end{proof}

Now recall that $f$ is \textit{flat at infinity} whenever $$ \underset{d(y,0) \to \infty}{\lim\limits_{d(x,0) \to \infty}}  \dfrac{d(f(x),f(y))}{d(x,y)} =0.$$
It is rather easy to see that flateness at infinity is stronger that Property $(ii)$ from Proposition~\ref{spectrumunboundedcase}, which in turn is stronger than radial flateness. For more details we refer the reader to \cite[Remark~2.7]{ACP21}.

\begin{theorem} \label{thmFlatInf}
	Suppose that $f$ is flat at infinity. If $\widehat{f}(\gamma) = \lambda \gamma$ for some $\lambda \in \C \setminus \{0\}$ and $\gamma \in \F(M) \setminus \{0\}$, then $\supp(\gamma)$ is bounded. 
\end{theorem}

\begin{proof}
    For $n\in \mathbb{N}$, we define the map $w_n : M \to \R_+$ by
    \begin{equation*}
    \label{eq:T_h}
        w_n(x)=\begin{cases}
        0 & \text{if }  d(x,0) < 2^n \\
        2^{-n}d(x,0)-1 & \text{if } 2^n \leq  d(x,0)\leq 2^{n+1} \\
        1 & \text{if } d(x,0) > 2^{n+1}.
    \end{cases} \,
    \end{equation*}
    This map is $2^{-n}$-Lipschitz. Moreover $T_n := w_n \widehat{Id}$ defines a bounded operator on $\F(M)$ with $\|T_n\| \leq 4$. A proof of the latter fact can be found in \cite[Section~2.2]{AP23}.
    It is rather easy to see that $\supp\left(T_n(\gamma) \right) \subset M \setminus B(0, 2^n)$ for any $\gamma\in \F(M)$. Therefore we can see $T_n$ as a bouded operator from $\F(M)$ to $\F_M(M \setminus B(0, 2^n))$.

    Let $\lambda \in \C$ be an eigenvalue of $\widehat{f}$ associated to an eigenvector $\gamma$ with unbounded support. Let us fix $\ep>0$. Since $f$ is flat at infinity, $f$ is also radially flat. Hence, by Lemma~\ref{lemma:Rad}, there exists $R_0>0$ such that $\Lip(f\restricted_{\{0\} \cup M \setminus B(0,R_0)}) < \ep$ and $f(\B(0,R)) \subset \B(0 , \frac{R}{2})$ for every $R \geq R_0$. We fix $n \in \N$ such that $2^n \geq R_0$. For simplicity, we will write $C_n = M\setminus B(0,2^n)$.
    We have $T_n \gamma \in \F_M(C_n)$ and since the support is unbounded, $T_n \gamma \neq 0$. We let $g_n \in \Lip_0(\{0\} \cup C_n)$ of norm $1$ be norming for $T_n \gamma$. The function $h_n = T_n^*g_n = w_n g_n$ is supported in $C_n$.
    Next, write $\widehat{f}(\gamma) = \widehat{f}(T_n\gamma) + \widehat{f}((I-T_n)\gamma)$ and notice that $(I-T_n)(\gamma) = (1-w_n)\widehat{Id}(\gamma)$ is supported in $B(0,2^{n+1})$ (because $(1-w_n)$ is). In particular $\widehat{f}((I-T_n)\gamma)$ is supported in $\overline{f(\B(0, 2^{n+1}))} \subset \B(0,2^n)$. Hence,
    $
    \langle h_n, \widehat{f}((I-T_n)\gamma) \rangle = 0
    $
    which implies that $
    \langle h_n, \widehat{f}(\gamma) \rangle = \langle h_n, \widehat{f}(T_n(\gamma)) \rangle.
    $
    
To conclude, we note that
\begin{align*}
|\lambda| \|T_n \gamma\| = |\lambda \langle g_n, T_n \gamma \rangle |
 = |\langle h_n, \widehat{f}(\gamma) \rangle |
& = | \langle h_n, \widehat{f}(T_n(\gamma)) \rangle | \\
& \leq \|h_n\|_L \; \|f\restricted_{\{0\} \cup C_n} \|_L \|T_n \gamma\| \\
& \leq 4\ep \|T_n \gamma\|.
\end{align*}
This implies that $|\lambda | \leq 4\ep$. Since $\ep$ was arbitrary, we obtain $\lambda = 0$.
\end{proof}

Since when $\widehat{f} : \F(M) \to \F(M)$ is compact, $f$ is automatically radially flat (see Lemma~2.4 in \cite{ACP21}), we obtain the following corollary. 

\begin{corollary}
	\label{cor-spectrum-flat-at-infinity}
	Let $M$ be a complete metric space and $f : M \to M$ be a base point preserving Lipschitz map. Let $A$ be the set of all natural numbers $n$ such that $f$ has a periodic point of order $n$. If $f$ is flat at infinity and $\widehat{f} : \F(M) \to \F(M)$ is compact, then $A$ is finite and moreover
	$$\sigma(\widehat{f})\setminus\{0\} = \bigcup_{n \in A} \U_n .$$
\end{corollary}

\begin{proof}
	If $\widehat{f}(\gamma) = \lambda \gamma$ for some $\lambda \in \C \setminus \{0\}$ and $\gamma \in \F(M) \setminus \{0\}$, then $\supp(\gamma)$ is bounded by the previous theorem. Now, thanks to Lemma~\ref{lemma:Rad}, we may choose $R>0$ large enough so that $\B(0,R)$ contains $\supp(\gamma)$ and $f(\B(0,R)) \subset \B(0,R)$. Now apply Corollary~\ref{cor-spectrum-compact} to $\widehat{f\restricted_{\B(0,R)}}$ in order to get the desired conclusion. 
\end{proof}

\section{Final remarks and open questions}

\subsection{When the metric space is uniformly discrete and bounded}

Recall that a metric space $M$ is said to be uniformly discrete if there exists $\delta >0$ such that, $\forall x \neq y \in M$, $d(x,y) > \delta$. If $M$ is moreover bounded and separable (that is countable), then the sequence $(\delta(x))_{x \in M \setminus \{0\}}$ is a Schauder basis for $\F(M)$ (which is equivalent to the $\ell_1$-basis). On such a space, any map $f : M \to M$ is Lipschitz. If $w : M \to \mathbb{C}$ is a weight map and $f(0)=0$ or $w(0)=0$, then, it is readily seen that $w \widehat{f}$ is bounded if and only if $w$ is bounded.
\smallskip

In this setting, it is easy to check whether an element $\gamma\in \F(M)$ is zero by simply looking at its coefficients in the Schauder basis. In particular, contrary to the general case (see Remark~\ref{remark:injectivity}), it is much simpler to describe precisely when 0 belongs to the point spectrum:

\begin{proposition} \label{Prop-injectivity}
    Let $M = \{0\} \cup \{x_i \; : \; i \in \N\}$ be uniformly discrete, bounded and separable. Let $f : M \to M$ be any map and $w : M \to \C$ be a bounded weight such that either $f(0)=0$ or $w(0)=0$. Then
\begin{align*}
    0 \in \sigma_p(w\widehat{f})  \ \Longleftrightarrow & \ f \ \text{is not injective or there exists} \ x\in M\setminus \{0\} \ \text{such that either } f(x)=0  \text{ or }  w(x)=0.
\end{align*}
\end{proposition}

\begin{proof}
``$\impliedby$" If $x\neq 0_M$ is such that $f(x)=0$ or $w(x)=0$ then $\delta(x)$ is obviously an eigenvector associated with the eigenvalue $0$. Assume now that $f$ and $w$ do not vanish outside of $\{ 0 \}$ and that $f$ is not injective. Then, there exist $x\neq y$, $y\neq 0$, such that $f(x) = f(y)$. Let $\gamma = \frac{w(x)}{w(y)} \delta(y) - \delta(x)$. Then $\gamma$ is nonzero and satisfies
$$
w\widehat{f}(\gamma) =  \frac{w(x)}{w(y)} w(y)\delta(f(y)) - w(x) \delta(f(x)) = w(x)\delta(f(x)) - w(x)\delta(f(x)) = 0, \ \text{  that is } 0 \in \sigma_p(w\widehat{f}).
$$

\noindent ``$\implies$" Assume that $0 \in \sigma_p(\omega\widehat{f})$ and that $f$ is injective. Let $\gamma = \sum_{i=1}^{+\infty} a_i \delta(x_i) \in \ker(w\widehat{f})$ with $\gamma\neq 0$. Then $w\widehat{f}(\gamma) =  \sum_{i=1}^{+\infty} a_i w(x_i) \delta(f(x_i)) = 0$. Since the $f(x_i), i \in \mathbb{N}$, are pairwise distinct, we deduce that $a_iw(x_i)\delta(f(x_i))=0$ for every $i\in \mathbb{N}$. Since $\gamma \neq 0$, there is $k$ such that $a_k \neq 0$ which implies that $f(x_k)=0$ or $w(x_k)=0$.
\end{proof}

Let us point out that, in this setting, if $w \equiv 1$ then $\widehat{f}$ is injective if and only if $f$ is so. The similar statement for surjectivity is true; More information can be found in \cite{GPP22}. Characterizing the surjectivity in the weighted case is slightly more intricate. Notice first that the weighted version of \cite[Proposition~2.1]{ACP20} holds true (in full generality), that is:
    $$  w\widehat{f} \text{ has dense range} \iff f(\mathrm{coz}(w)) \cup \{0\} \ \text{is dense in } M.$$
When the metric space $M$ is uniformly discrete, the statement on the right-hand side naturally becomes $f(\mathrm{coz}(w)) \cap M \setminus \{0\} = M \setminus \{0\}$. In particular, if $f(0) = 0$ and $w \widehat{f}$ is surjective then $f$ is also surjective. Furthermore, one has the next characterization of surjective weighted Lipschitz operators.

\begin{proposition}\label{Prop-surjectivity}
    Let $M = \{0\} \cup \{x_i \; : \; i \in \N\}$ be uniformly discrete, bounded and separable. Let $f : M \to M$ be any map and $w : M \to \C$ be a bounded weight such that either $f(0)=0$ or $w(0)=0$. Then the following are equivalent:
    \begin{enumerate}
        \item $w\widehat{f}$ is surjective.
        \item $M\setminus \{0\} \subset f(M)$ and:
            $ \exists c>0, \; \forall y \in M\setminus \{0\}, \; \exists x \in f^{-1}(\{y\}) \text{ such that } |w(x)|>c. $
    \end{enumerate}
\end{proposition}
\begin{proof} 
    $(1) \impliedby (2)$. Let $\gamma \in \F(M)$. Then there exists (a unique) $(b_n)_n \in \ell_1$ such that $\gamma = \sum_n b_n \delta(x_n)$.  By assumption, for every $n\in \mathbb{N}$, $f^{-1}(\{x_n\}) \neq \emptyset$ and there exists $y_n \in f^{-1}(\{x_n\})$ such that $|w(y_n)| > c$. Define, for every $n$, $a_n = \frac{b_n}{w(y_n)}$. The inequality $|a_n| \leq \frac{|b_n|}{c}$ implies that $(a_n)_n \in \ell_1$ so that $\mu := \sum_n a_n\delta(y_n) \in \F(M)$. This element clearly satisfies $w\widehat{f}(\mu)=\gamma$, which proves the surjectivity of $w\widehat{f}$.

    $(1) \implies (2)$. Assume that $w\widehat{f} : \F(M) \rightarrow \F(M)$ is surjective. It is a classical fact that in such case, its adjoint operator $wC_f \ \Lip_0(M) \rightarrow \Lip_0(M)$ is an isomorphism onto its range, see \cite[Theorem 4.15]{Rudin}. 
    In particular, there exists a constant $C>0$ such that, for every $g\in \Lip_0(M)$, $\|g\|_L \leq C \|wC_f(g)\|_L$. Note that on such metric space, any bounded function is Lipschitz and $\| \cdot \|_L$ is equivalent to $\| \cdot \|_{\infty}$. Hence, the latter inequality reads $\|g\|_{\infty} \leq C' \|wC_f(g)\|_{\infty}$ for some constant $C'>0$. Let $n\in \mathbb{N}$ and define $g_n$ on $M$ by setting $g_n(0)=0$, $g_n(x_n)=1$ and $g_n(x_i)=0$ if $i\neq n$. 
    Then one has $wC_f(g_n)(x)=w(x)$ if $x\in f^{-1}(\{x_n\})$, and $0$ otherwise. In particular, we have
    $$
    1 \leq C' \Big( \underset{x\in f^{-1}(\{x_n\})}{\sup} \ |w(x)| \Big),
    $$
    which yields the existence of $x\in f^{-1}(\{x_n\})$ such that $|w(x)| > \frac{1}{2C'}=: c$, and concludes the proof.
\end{proof}

As a direct consequence of Proposition \ref{Prop-injectivity} and Proposition \ref{Prop-surjectivity}, we can characterize the belonging of $0$ to the spectrum of $w\widehat{f}$. In what follows, $\rho(w\widehat{f})$ denotes the resolvent set of $w\widehat{f}$.

\begin{corollary}\label{caraciso}
Let $M = \{0\} \cup \{x_i \; : \; i \in \N\}$ be uniformly discrete, bounded and separable. Let $f : M \to M$ be any map and $w : M \to \C$ be a bounded weight such that either $f(0)=0$ or $w(0)=0$. Then the following are equivalent:
    \begin{enumerate}
        \item $0 \in \rho(w\widehat{f})$, that is $w\widehat{f}$ is an isomorphism;
        \item $f$ is injective, $f(M\setminus \{0\}) = M\setminus \{0\}$, $w(M\setminus \{0\}) \subset \C\setminus \{0\}$ and $\underset{x\in M\setminus \{0\}}{\inf} \ |w(x)| > 0$.
    \end{enumerate}
\end{corollary}

\begin{remark}
In fact, we can derive from Proposition $3.9$ and Proposition $3.10$ from \cite{ACP23} the following characterization, which holds true for any pointed metric space $M$. The following are equivalent:
\begin{enumerate}
    \item $w\widehat{f}$ is an isomorphism;
    \item $wC_f$ is an isomorphism;
    \item $w$ does not vanish on $M \setminus \{0\}$, $f(M) \cup \{0\}$ is dense in $M$, $f$ is injective and the operator defined by
\[
  T \colon\begin{aligned}[t]
        \F(f(M)) &\longrightarrow \F(M)\\
        \delta(f(x)) &\longmapsto  \frac{1}{w(x)} \delta(x)
  \end{aligned} \ \ \ \text{ is bounded.}
\]
\end{enumerate}
Note that if $f(M) \cup \{0\}$ is dense in $M$, the space $\F(f(M))$ 
is dense in $\F(M)$ so that, when $T$ is bounded, it extends to a bounded operator on $\F(M)$ which is precisely the inverse of $w\widehat{f}$. Of course, when $M$ uniformly discrete, bounded and separable, $T$ is bounded if and only if $\frac{1}{w}$ is bounded on $M \setminus \{0\}$, that is $\underset{x\in M\setminus \{0\}}{\inf} \ |w(x)| > 0$, so we retrieve the result of Corollary \ref{caraciso}.
\end{remark}

When $M=\mathbb{N}\bigcup \{0\}$ (respectively $M = \Z$) with an appropriate metric $d$, by setting $f(0)=0$ and $f(n)=n-1$ one obtains that $\widehat{f}$ is conjugate to a unilateral weighted shift on $\ell_1$ (respectively a bilateral weighted shift on $\ell_1$); See Example \ref{Exshift} and, more generally \cite[Proposition 1.6]{ACP20}. In fact, the class of weighted Lipschitz operators $w\widehat{f} : \F(M) \to \F(M)$, for $M$ uniformly discrete bounded and separable, includes all kind of unilateral or bilateral weighted shifts on $\ell_1$. For such operators, a lot is known about the spectrum. We refer for instance to \cite{Daniello} and the references therein for results on the spectrum of unilateral and bilateral weighted shifts. In particular, \cite[Proposition 3.4]{Daniello} gives a precise localization of the point spectrum in the unilateral case.
\smallskip

The next result can be viewed as a localization of the point spectrum of $w\widehat{f}$. 
In what follows, $\D$ is the open unit disc in $\C$, and $\overline{\D}$ its closure.

\begin{proposition}\label{localizespectrum}
    Let $M = \{0\} \cup \{x_i \; : \; i \in \N\}$ be uniformly discrete, bounded and separable. Let $f : M \to M$ be a map,  let $w : M \to \C$ be a bounded weight and assume that $f(0)=0$ or $w(0)=0$. Then the point spectrum of $w\widehat{f}$ satisfies the following properties:
    \begin{enumerate}
        \item[$1.$] $\sigma_p(w\widehat{f}) \subset \overline{B}\big(0,\|w\|_{\infty}\big)$.
        \item[$2.$] If $\lambda \in \sigma_p(w\widehat{f}) \setminus \{0\}$ admits an associated eigenvector with finite support, then $|\lambda| \geq \inf_{x \in \supp(\gamma)} |w(x)|$. 
        \item[$3.$] If $f$ is injective and $\lambda \in \sigma_p(w\widehat{f}) \setminus \{0\}$ is an eigenvalue with no finitely supported eigenvector, then  $\inf_{x \in \supp(\gamma)} |w(x)| < |\lambda| <  \sup_{x \in \supp(\gamma)} |w(x)|$, for any eigenvector $\gamma$ associated with $\lambda$.\\
        In particular, if $w \equiv 1$ then there must be a finitely supported eigenvector associated with $\lambda$, and $\lambda$ is a root of unity.
    \end{enumerate}
\end{proposition}

\begin{proof}
First of all, notice that if $\gamma$ is an eigenvector of finite support associated with an eigenvalue $\lambda \neq 0$, then Lemma~\ref{lemmaPeriodic2} readily implies that 
$$ \inf_{x \in \supp(\gamma)} |w(x)| \leq |\lambda| \leq  \sup_{x \in \supp(\gamma)} |w(x)|,$$
which proves $2$.

$1.$  Assume that $\lambda \neq 0$ is an eigenvalue and $\gamma = \sum_{i=1}^{+\infty} a_i \delta(x_i)$ is an associated eigenvector with infinite support. 
Without loss of generality, we may assume that $a_i \neq 0$ for every $i \in \N$.
Then $w\widehat{f}(\gamma) = \lambda \gamma$ reads as follows:
$$ \lambda  \sum_{i=1}^{+\infty} a_i \delta(x_i) =   \sum_{i=1}^{+\infty} a_i w(x_i) \delta(f(x_i)). $$
Since $(\delta(x_i))_i$ is a Schauder basis, 
if we let $E_i := \{j \in \N \; : \; f(x_j) = x_i\}$ then we may write 
$$\lambda a_i\delta(x_i) = \sum_{j \in E_i} a_j w(x_j) \delta(f(x_j)) = \Big(\sum_{j \in E_i} a_j w(x_j) \Big) \delta(x_i).$$ Therefore 
$$ |\lambda| \; |a_i| = |\sum_{j \in E_i} a_j w(x_j)| \leq \Big(\sum_{j \in E_i} |a_j|\Big) \|w\|_{\infty}.$$
Summing the latter over $i \in \N$ yields
$$ |\lambda| \sum_{i=1}^{+\infty}|a_i| \leq  \sum_{i=1}^{+\infty} \Big( \sum_{j \in E_i} |a_j| \Big) \|w\|_{\infty} = \Big(\sum_{k=1}^{+\infty}|a_k| \Big) \|w\|_{\infty} \ \ \text{ which implies } | \lambda | \leq \|w\|_{\infty}. $$

$3.$ Next, suppose that $\lambda \neq 0$ is an eigenvalue with no finitely supported eigenvector. Let us fix an eigenvector $\gamma = \sum_{i=1}^{+\infty} a_i \delta(x_i) \in \F(M)$ with infinite support. Again, we can assume that $a_i \neq 0$ for every $i\in \mathbb{N}$, thus $\supp(\gamma)=M$.
Since $f$ is injective, each $E_j$ defined above is reduced to a point, so there exists $\sigma : \N \to \N$ bijective such that $a_{\sigma(i)} w(x_{\sigma(i)}) = \lambda a_i$ and $f(x_{\sigma(i)}) = x_i$. Applying this identity recursively, we obtain that for every $n\in \mathbb{N}$,
\begin{equation}\label{recursiveequation1}
a_{\sigma^n(1)}w(x_{\sigma^n(1)}) w(x_{\sigma^{n-1}(1)}) \cdots w(x_{\sigma(1)}) = \lambda^n a_1.
\end{equation}
Note that $\{  \sigma^n(1) \mid n\in \mathbb{N} \cup \{ 0 \} \}$ is infinite, otherwise, choosing $N \in \mathbb{N}$ minimal such that $\sigma^N(1) = 1$, the vector $\tilde{\gamma} = \sum_{k=1}^N a_{\sigma^k(1)}\delta(x_{\sigma^k(1)})$ would satisfy
\begin{align*}
w\widehat{f}(\gamma) = \sum_{k=1}^{N} a_{\sigma^k(1)} w(x_{\sigma^k(1)})  \delta(f(x_{\sigma^k(1)}))
& = \sum_{k=1}^{N} \lambda a_{\sigma^{k-1}(1)}  \delta(x_{\sigma^{k-1}(1)}) 
 = \lambda \sum_{j=0}^{N-1} \lambda a_{\sigma^{j}(1)}  \delta(x_{\sigma^{j}(1)})  
= \lambda \gamma.
\end{align*}
That is, $\tilde{\gamma}$ is a finitely supported eigenvector associated with $\lambda$, which contradicts our assumption. Hence, $\{  \sigma^n(1) \mid n\in \mathbb{N} \cup \{ 0 \} \}$ is infinite (which forces its elements to be pairwise distinct) so, since $(a_{\sigma^n(1) })_n \in \ell_1$, we must have $\lim_n a_{\sigma^n(1)} = 0$. To conclude, we use \eqref{recursiveequation1} to get $|\lambda|^n |a_1| \leq \|w\|_{\infty}^n |a_{\sigma^n(1)} |$. In other words:
$$
\left( \dfrac{| \lambda |}{\|w\|_{\infty}} \right)^n \leq \dfrac{|a_{\sigma^n(1)} |}{|a_1|} \underset{n\to +\infty}{\longrightarrow} 0, \ \ \text{ which implies that } | \lambda | < \|w\|_{\infty}.
$$

For the other inequality, there is nothing to prove if $\inf_{x \in M} |w(x)|=0$ since $\lambda \neq 0$. Assume that $\inf_{x \in M} |w(x)| >0$. Let $\mu = \sigma^{-1} : \mathbb{N} \to \mathbb{N}$ be the inverse of the bijecton $\sigma$ defined above. The equality $a_{\sigma(i)} w(x_{\sigma(i)}) = \lambda a_i$ can be rewritten $a_iw(x_i) = \lambda a_{\mu(i)}$ for every $i$, which implies that for any $n$,
$$
a_1 w(x_1) w(x_{\mu(1)}) \cdots w(x_{\mu^{n-1}(1)}) = \lambda^n a_{\mu^n(1)}.
$$
We deduce that
$$
\left( \dfrac{\inf_{x \in M} |w(x)|}{|\lambda|} \right)^n \leq \dfrac{|a_{\mu^n(1)}|}{|a_1|} \underset{n\to +\infty}{\longrightarrow} 0,
$$
which yields the desired inequality and concludes the proof.

In particular, if $w \equiv 1$ and $\lambda \in \sigma_p(\widehat{f}) \setminus \{0\}$, the inequality $\inf_{x \in \supp(\gamma)} |w(x)| < |\lambda| <  \sup_{x \in \supp(\gamma)} |w(x)|$ becomes $1<|\lambda|<1$, which is impossible, so there must be a finitely supported eigenvector for $\lambda$. By Lemma~\ref{lemmaPeriodic2}, this implies that $\lambda$ is a root of unity.
\end{proof}

\subsection{Open questions}

We conclude this note with two open questions.

\begin{question}
    Does Corollary~\ref{cor-spectrum-compact} still hold when $M$ is unbounded? 
	That is, if $\widehat{f} : \F(M) \to \F(M)$ is compact, then is it true that the set $A$, of all natural numbers $n$ such that $f$ has a periodic point of order $n$, is finite, and moreover
	$$\sigma(\widehat{f})\setminus\{0\} = \bigcup_{n \in A} \U_n ?$$
\end{question}

\begin{question} Similarly, can one obtain
a similar description of the the spectrum of compact weighed Lipschitz operators $w \widehat{f}$ with weaker assumptions on $f$ than in Theorem \ref{thm-spectrum}?
\end{question}

\noindent \textbf{Acknowledgements :} The third author was supported by the French ANR project No. ANR-20-CE40-0006.

\end{document}